\documentclass[preprint,12pt]{elsarticle}

\usepackage{enumerate,amsmath,amsthm,amsopn,amstext,amscd,amsfonts,amssymb}
\usepackage{color}

\renewcommand{\a}{\alpha}

\renewcommand{\d}{\delta}
\newcommand{\D}{\Delta}

\newtheorem{theorem}{Theorem}[section]
\newtheorem{lemma}[theorem]{Lemma}

\newtheorem{corollary}[theorem]{Corollary}

\journal{-}

\begin{document}

\begin{frontmatter}



\title{New inequalities involving the Geometric-Arithmetic index\tnoteref{Proyectos}}
 \tnotetext[Proyectos]{Supported in part by two grants from Ministerio de Econom{\'\i}a y Competititvidad (MTM2013-46374-P and MTM2015-69323-REDT), Spain,
and a grant from CONACYT (FOMIX-CONACyT-UAGro 249818), M\'exico.}


\author[Jose]{Jos\'e M. Rodr{\'\i}guez}
\address[Jose]{Departamento de Matem\'aticas, Universidad Carlos III de Madrid,
Avenida de la Universidad 30, 28911 Legan\'es, Madrid, Spain}
\ead{jomaro@math.uc3m.es}
\author[JuanAlberto]{Juan A. Rodr{\'\i}guez-Vel\'azquez\corref{cor1}}
\cortext[cor1]{Corresponding author}
\address[JuanAlberto]{
Departament d'Enginyeria
Inform\`{a}tica i Matem\`{a}tiques, Universitat Rovira i Virgili, 
Av. Pa\"{i}sos Catalans 26, 43007
Tarragona, Spain}
\ead{juanalberto.rodriguez@urv.cat}
\author[Sigarreta]{Jos\'e M. Sigarreta}
\address[Sigarreta]{Facultad de Matem\'aticas, Universidad Aut\'onoma de Guerrero,
Carlos E. Adame No.54 Col. Garita, 39650 Acalpulco Gro., Mexico}\ead{jsmathguerrero@gmail.com}

\begin{abstract}
Let $G=(V,E)$ be a simple connected graph and $d_i$ be the degree of its $i$th vertex. In a recent paper
[J. Math. Chem. 46 (2009) 1369-1376]
the \emph{first geometric-arithmetic index} of a graph $G$ was defined as
$$GA_1=\sum_{ij\in E}\frac{2 \sqrt{d_i d_j}}{d_i + d_j}.$$
This graph invariant is useful for chemical proposes. The main use of $GA_1$ is for designing so-called quantitative structure-activity relations and quantitative structure-property relations.
%, where ``structure" means  molecular structure, ``property" some physical or chemical property and ``activity" some biologic, pharmacologic or similar property. 
In this paper we obtain new inequalities involving the geometric-arithmetic index $GA_1$
and characterize the graphs which make the inequalities tight.
In particular, we improve some known results,  generalize other,  and we relate $GA_1$ to other well-known topological indices.
\end{abstract}


\end{frontmatter}


\section{Introduction}
\label{Introduction}
A graph invariant is a property of graphs that is preserved by isomorphisms. Around the middle of the last century theoretical chemists discovered that some interesting  relationships between various properties of organic substances and the molecular structure  can be deduced by examining some invariants of the underlining molecular graph. Those graph invariants that are useful for chemical purposes were named \emph{topological indices} or \emph{molecular structure descriptors}.  The Wiener index, introduced by Harry Wiener in 1947, is the oldest topological index related to molecular branching.  Wiener defined this topological  index   as the sum of all shortest-path distances of a graph, and he showed that it is closely correlated with the boiling points of alkane molecules \cite{Wi}. Based on its success, many other topological indices have been developed subsequently to Wiener's work.

Topological indices based on vertex degrees  have been
used over 40 years. Among them, several indices are recognized to be useful tools in
chemical researches. Probably, the best known such descriptor is the Randi\'c connectivity
index  \cite{R}. There are more than thousand papers and a couple of books dealing with
this molecular descriptor (see, e.g., \cite{GF,LG,LS,RS,RS0} and the references therein).
During many years, scientists were trying to improve the predictive power of the
Randi\'c index. This led to the introduction of a large number of new topological
descriptors resembling the original Randi\'c index.
The first geometric-arithmetic index $GA_1$, defined in \cite{VF} as
$$
GA_1(G) = \sum_{uv\in E(G)}\frac{2\sqrt{d_u d_v}}{d_u + d_v}
$$
where $uv$ denotes the edge of the graph $G$ connecting the vertices $u$ and $v$, and
$d_u$ is the degree of the vertex $u$,
is one of the successors of the Randi\'c index.
Although $GA_1$ was introduced in $2009$, there are many papers dealing with this index
(see, e.g., \cite{D,DGF,DGF2,MH,DasTrinajstic2010,DZT,RS2,RS3,S,VF} and the references therein).
There are other geometric-arithmetic indices, like $Z_{p,q}$ \cite{DGF}, where $Z_{0,1} = GA_1$, but the results in \cite[p.598]{DGF}
show that $GA_1$   gathers the
same information on observed molecule as   $Z_{p,q}$.

As described in \cite{DGF}, the reason for introducing a new index is to gain prediction of target property (properties)
of molecules somewhat better than obtained by already presented indices. Therefore,
a test study of predictive power of a new index must be done. As a standard for
testing new topological descriptors, the properties of octanes are commonly used.
We can find 16 physico-chemical properties of octanes at www.moleculardescriptors.eu. The $GA_1$ index gives better correlation coefficients than the Randi\'c index for these properties, but the differences between
them are not significant. However, the predicting ability of the $GA_1$ index compared with
Randi\'c index is reasonably better (see \cite[Table 1]{DGF}).
Although only about 1000 benzenoid hydrocarbons are known, the number of
possible benzenoid hydrocarbons is huge. For instance, the number of
possible benzenoid hydrocarbons with 35 benzene rings is $5.85\cdot 10^{21}$ \cite{NGJ}.
Therefore, the modeling of their physico-chemical properties is very important in order
to predict properties of currently unknown species.
The graphic in \cite[Fig.7]{DGF} (from \cite[Table 2]{DGF}, \cite{TRC}) shows
that there exists a good linear correlation between $GA_1$ and the heat of formation of benzenoid hydrocarbons
(the correlation coefficient is equal to $0.972$).
Furthermore, the improvement in
prediction with $GA_1$ index comparing to Randi\'c index in the case of standard
enthalpy of vaporization is more than 9$\%$. That is why one can think that $GA_1$ index
should be considered  for designing so-called quantitative structure-activity relations and quantitative structure-property relations, where ``structure" means  molecular structure, ``property" some physical or chemical property and ``activity" some biologic, pharmacologic or similar property.

Some inequalities involving 
the geometric-arithmetic index   and other topological indices were obtained in \cite{D, DGF,DGF2,DasTrinajstic2010,MH,RS2,RS3,S,VF}. The aim of this paper is to obtain new inequalities involving the geometric-arithmetic index $GA_1$
and characterize the graphs which make the inequalities tight.
In particular, we improve some known results,  generalize other,  and we relate $GA_1$ to other well-known topological indices.

\section{New equalities involving $GA_1$}

Throughout this paper, $G=(V,E)=(V (G),E (G))$ denotes a (non-oriented) finite simple (without multiple edges and loops) connected graph with $E \neq \emptyset$.
Note that the connectivity of $G$ is not an important restriction, since if $G$ has connected components $G_1,\dots,G_r,$ then
$$GA_1(G) = GA_1(G_1) + \cdots + GA_1(G_r).$$ Furthermore, every molecular graph is connected.

From now on, the order (the cardinality of $V(G)$), size (the cardinality of $E(G)$), and maximum and minimum degree  of $G$ will be denoted by $n,m,\D,\d,$ respectively.

\smallskip

We will denote by $M_1$ and $M_2$ the first and second Zagreb indices, respectively, defined as
$$
M_1 = M_1(G) = \sum_{u\in V(G)} d_u^2,
\qquad
M_2 = M_2(G) = \sum_{uv\in E(G)} d_u d_v .
$$
These topological indices have attracted growing interest, see e.g., \cite{BF,D2,FGE,GT,L} (in particular, they are included in a number of programs used for the routine computation of topological indices).

%In \cite{DGF2} (see also \cite[p.611]{DGF}) we find the bound
%\begin{equation} \label{eq1}
%GA_1(G) \le \frac{ \sqrt{m M_2(G)} }{\d}\,.
%\end{equation}
%
%
%The following result gives a lower bound for $GA_1$ similar to \eqref{eq1}.

The following inequality was given in 
 \cite{MH} (see also \cite[p.610]{DGF}) and \cite[Theorem 3.7]{RS2},
\begin{equation} \label{eq24}
GA_1(G) \le \frac1{2\d}\, M_1(G) .
\end{equation}
Since $M_1(G) \ge \d^2 n$,
Theorem \ref{t:end2} below improves \eqref{eq24}.

%In \cite{MH} (see also \cite[p.610]{DGF}), \cite[Theorem 3.7]{RS2} and \cite[Theorem 2]{S} appear the inequalities
%\begin{equation} \label{eq24}
%GA_1(G) \le \frac1{2}\, M_1(G),
%\qquad
%GA_1(G) \le \frac1{2\d}\, M_1(G),
%\qquad
%GA_1(G) \le \frac{\sqrt{mM_1(G)}}{2\d}\,  .
%\end{equation}
%Since $M_1(G) \ge \d^2 n$ and ??,
%Theorem \ref{t:end2} below improves \eqref{eq24}.

\begin{theorem} \label{t:end2}
For any graph $G$,
$$
\frac{\d M_1(G)}{2 \D^2} \le GA_1(G) \le \frac{\sqrt{nM_1(G)}}{2} \, ,
$$
and each equality holds if and only if $G$ is a regular graph.
\end{theorem}

\begin{proof}
First of all, note that for every function $f:[\d,\D] \rightarrow \mathbb{R}$, we have
$$
\sum_{uv\in E(G)} \left(f(d_u) + f(d_v)\right)
= \sum_{u\in V(G)} d_u f(d_u) .
$$
Since $\frac{4}{d_u + d_v}\le \frac1{d_u} + \frac1{d_v}$, by taking $f(d_u)=\frac{1}{d_u}$ we deduce
$$
\sum_{uv\in E(G)}\frac{1}{d_u + d_v}
\le \frac14 \sum_{uv\in E(G)}  \left(\frac1{d_u} + \frac1{d_v}\right)
= \frac{n}{4} \,.
$$

Cauchy-Schwarz inequality gives
$$
\begin{aligned}
GA_1(G) & = \sum_{uv\in E(G)}\frac{2\sqrt{d_u d_v}}{d_u + d_v}\\
&\le
\sum_{uv\in E(G)}\sqrt{d_u+d_v}\frac{1}{\sqrt{d_u+d_v}}
\\
& \le
\left(\sum_{uv\in E(G)}(d_u + d_v)\right)^{1/2} \left(\sum_{uv\in E(G)}\frac{1}{d_u + d_v}\right)^{1/2}
\\
& \le \left(\sum_{u\in V(G)} d_ud_u\right)^{1/2} \left(\frac{n}{4}\right)^{1/2}\\
&= \frac{\sqrt{nM_1(G)}}{2}
\,.
\end{aligned}
$$

On the other hand, for any $uv\in E(G)$ we have
$$
\frac{2\d}{d_u + d_v}
\le \frac{2\sqrt{d_u d_v}}{d_u + d_v}
$$
and
$$
\frac1{\D^2}
\le \frac{4}{(d_u + d_v)^2}
= \frac{\frac{2}{d_u + d_v}}{\frac{d_u + d_v}{2}} \, ,
$$
Hence,
$$
\frac1{\D^2}\, \frac{d_u + d_v}{2}
\le \frac{2}{d_u + d_v} \,,
$$
and so
$$
\frac{\d}{\D^2}\, \frac{d_u + d_v}{2}
\le \frac{2\d}{d_u + d_v}
\le \frac{2\sqrt{d_u d_v}}{d_u + d_v}, \, 
$$
which implies that
$$\frac{\d M_1(G)}{2\D^2}=\frac{\d}{2\D^2}\sum_{u\in V(G)} d_u^2=\frac{\d}{2\D^2}\sum_{uv\in E(G)} (d_u+d_v)
\le GA_1(G).
$$
Notice that all the equalities above hold if and only if the graph is regular. 
%Moreover, if the graph is regular, then $GA_1(G)=m$, $\d M_1(G)=2\d^2m$ and $nM_1(G)=n^2\d^2=(2m)^2$, and both equalities hold. 
%Now, if the equality holds on the right hand side, then $2\sqrt{d_u d_v}=d_u + d_v$ for every $uv\in E(G)$; hence, Corollary \ref{c:t} gives $d_u = d_v$ for every $uv\in E(G)$ and, since $G$ is connected, the graph is regular.
%Finally, if the equality holds on the left hand side, then $\sqrt{d_u d_v}=\d$ for every $uv\in E(G)$; hence, Corollary \ref{c:t} gives $d_u = d_v$ for every $uv\in E(G)$ and so $G$ is regular.
\end{proof}

The following elementary lemma will be an important tool to derive some results.

%\begin{lemma} \label{l:t}
%Let $f$ be the function $f(t)=\frac{2t}{1+t^2}$ on the interval $[0,\infty)$. Then $f$
%strictly increases in $[0,1]$, strictly decreases in $[1,\infty)$, $f(t)=1$ if and only if $t=1$ and $f(t)=f(t_0)$ if and only if either
%$t=t_0$ or $t=t_0^{-1}$.
%\end{lemma}
%
%The following consequence of Lemma \ref{l:t} will be very useful.

\begin{lemma} \label{c:t}
Let $g$ be the function $g(x,y)=\frac{2\sqrt{xy}}{x + y}$ with $0<a\le x,y \le b$. Then
$$
\frac{2\sqrt{ab}}{a + b} \le g(x,y) \le 1.
$$
The equality in the lower bound is attained if and only if either $x=a$ and $y=b$, or $x=b$ and $y=a$,
and the equality in the upper bound is attained if and only if $x=y$.
Besides, $g(x,y)= g(x',y')$ for some $x',y'>0$ if and only if $x/y$ is equal to either $x'/y'$ or $y'/x'$.
Finally, if $0 \le x'<x \le y$, then $g(x',y)< g(x,y)$.
\end{lemma}

%\begin{proof}
%It suffices to apply Lemma \ref{l:t}, since $g(x,y)=f(t)$ with $t=\sqrt{\frac{x}{y}}\,$, and $\sqrt{\frac{a}{b}}\le t %\le \sqrt{\frac{b}{a}}\,$.
%\end{proof}

\begin{theorem} \label{t:p4bis}
For any graph $G$,
$$
\frac{\sqrt{(\D+\d)^2 M_2(G)+4\D^3\d\, m(m-1)}}{\D(\D+\d)}
\le GA_1(G)
\le \frac{\sqrt{M_2(G)+\d^2 m(m-1)}}{\d} \, ,
$$and each equality holds if and only if $G$ is regular.
\end{theorem}

\begin{proof}
By  Lemma \ref{c:t}, for any edge $uv\in E(G)$ we have
\begin{equation}\label{EQ-1-M2}
\frac{2\sqrt{d_u d_v}}{d_u + d_v} \ge \frac{2\sqrt{\D\d}}{\D+\d}\; .
\end{equation}
Notice also that
\begin{equation}\label{EQ-2-M2}
\frac{1}{\Delta^2} \le \frac{4}{(d_u+d_v)^2}\le \frac{1}{\delta^2}
\end{equation}
Inequalities \eqref{EQ-1-M2} and \eqref{EQ-2-M2} lead to
$$
\begin{aligned}
\left( GA_1(G)\right)^2
& = \left( \sum_{uv\in E(G)}\frac{2 \sqrt{d_u d_v}}{d_u + d_v} \, \right)^2
\\
& = \sum_{uv\in E(G)}\frac{4 d_u d_v}{(d_u + d_v)^2}
+ 2\sum_{\begin{array}{c}
uv,xy\in E(G),
\\
uv\neq xy
\end{array}}\frac{2 \sqrt{d_u d_v}}{d_u + d_v}\,\frac{2 \sqrt{d_x d_y}}{d_x + d_y}
\\
& \ge \frac{1}{\D^2}\sum_{uv\in E(G)} d_u d_v + \sum_{\begin{array}{c}
uv,xy\in E(G),
\\
uv\neq xy
\end{array}} \frac{8 \D\d}{(\D+\d)^2}
\\
& = \frac{M_2(G)}{\D^2} + \frac{4 \D\d}{(\D+\d)^2}\,m(m-1)
\\
& = \frac{(\D+\d)^2 M_2(G)+4\D^3\d \,m(m-1)}{\D^2(\D+\d)^2} \, .
\end{aligned}
$$
In a similar way, we obtain
$$
\begin{aligned}
\left( GA_1(G)\right)^2
& = \sum_{uv\in E(G)}\frac{4 d_u d_v}{(d_u + d_v)^2}
+ 2\sum_{\begin{array}{c}
uv,xy\in E(G),
\\
uv\neq xy
\end{array}}\frac{2 \sqrt{d_u d_v}}{d_u + d_v}\,\frac{2 \sqrt{d_x d_y}}{d_x + d_y}
\\
& \le \frac{1}{\d^2}\sum_{uv\in E(G)} d_u d_v + 2\sum_{\begin{array}{c}
uv,xy\in E(G),
\\
uv\neq xy
\end{array}} 1
\\
& = \frac{M_2(G)}{\d^2} + m(m-1)
\\
& = \frac{M_2(G)+\d^2 m(m-1)}{\d^2} \, .
\end{aligned}
$$
To conclude the proof we can observe that  equality \eqref{EQ-1-M2} holds for $uv\in E(G)$ if and only if $d_u=\Delta$ and $d_v=\delta$ or $d_u=\delta$ and $d_v=\Delta$. Furthermore, the equalities in  \eqref{EQ-2-M2} hold for every $uv\in E(G)$ if and only if $G$ is regular.
\end{proof}

We will use the following particular case of Jensen's inequality.

\begin{lemma} \label{l:Jensen}
If $f$ is a convex function in $\mathbb{R}_+$ and $x_1,\dots,x_k> 0$, then
$$
f\left(\frac{ x_1+\cdots +x_k }{k} \right) \le \frac{1}{k} \, \left(f( x_1)+\cdots +f(x_k) \right) .
$$
\end{lemma}

We recall that  the Randi\'c index is defined as
$$
 R(G) = \sum_{uv\in E(G)} \frac1{\sqrt{d_u d_v}} \, .
$$
The following result provides 
a bound on $GA_1$ involving the Randi\'c index.

\begin{theorem} \label{t:r}
For any graph $G$,
$$
GA_1(G) + \D R(G)\geq 2m,
$$
and the equality holds if and only if $G$ is a regular graph.
\end{theorem}

\begin{proof}
It is well-known that for all $a,b >0$,
$$%\begin{equation} \label{eq:r}
\frac{a}{b}+\frac{b}{a}\geq 2,
$$%\end{equation}
and the equality holds if and only if $a=b$.
Applying this inequality, we obtain
%Applying \eqref{eq:r}, we obtain
$$
\sum_{uv\in E(G)}\frac{2\sqrt{d_u d_v}}{d_u + d_v}+ \sum_{uv\in E(G)}\frac{d_u + d_v}{2\sqrt{d_u d_v}}\geq 2m.
$$
Therefore,
$$
\sum_{uv\in E(G)}\frac{2\sqrt{d_u d_v}}{d_u + d_v}+ \sum_{uv\in E(G)}\frac{\D}{\sqrt{d_u d_v}}\geq 2m.
$$
and we have
$$
GA_1(G) + \D R(G)\geq 2m.
$$

To conclude the proof we only need to observe  that the above equality holds if and only if $2\sqrt{d_u d_v}=d_u + d_v=2\Delta$ for every $uv\in E(G)$.
\end{proof}

In 1998 Bollob\'{a}s and Erd\"{o}s \cite{BE} generalized the Randi\'{c} index by replacing $1/2$ by any real number. Thus, for $\alpha \in \mathbb{R}\setminus \{0\}$, the \emph{general Randi\'{c} index} is defined as 
$$
R_\a = R_\a(G) = \sum_{uv\in E(G)} (d_u d_v)^\a .
$$
The general Randi\'{c} index, also called \emph{variable Zagreb index} in 2004 by Mili{c}evi\'{c} and Nikoli\'{c}  \cite{MN}, has been extensively studied \cite{LG}.
Note that $R_{-1/2}$ is the usual Randi\'c index, $R_{1}$ is the second Zagreb index $M_2$,
$R_{-1}$ is the modified Zagreb index \cite{NKMT}, etc.
In Randi\'{c}'s original paper \cite{R}, in addition to the particular case $\alpha=-1/2$, also the index with $\alpha=-1$ was briefly considered.

\smallskip

Next, we will prove some bounds on $GA_1$ involving the general Randi\'{c} index. To this end, we need the following additional tool.

%In \cite[Lemma 3]{S} appears the following result.

\begin{lemma}{\rm \cite[Lemma 3]{S}} \label{l:h}
Let $h$ be the function $h(x,y)=\frac{2xy}{x+y}$ with $\d\le x,y \le \D$. Then
$$
\d \le h(x,y) \le \D.
$$
Furthermore, the lower (respectively, upper) bound is attained if and only if $x=y=\d$ (respectively, $x=y=\D$).
\end{lemma}

As we will show in Theorems \ref{t:z1} and \ref{t:lb55}, bounds on $R_{\a}$ immediately impose bounds on $GA_1$.

\begin{theorem} \label{t:z1}
Let $G$ be a graph  and $\a \in \mathbb{R} \setminus\{0\}$. Then the following statements hold. 
\begin{enumerate}[{\rm (a)}]
\item If $\a \le -1/2$, then $\d^{-2\a} R_{\a}(G) \leq GA_1(G)\leq \D^{-2\a} R_{\a}(G)$.
\item $\a \ge -1/2$, then $\d \D^{-2\a-1} R_{\a}(G) \leq GA_1(G)\leq \D \d^{-2\a-1} R_{\a}(G) $.
\end{enumerate}
Furthermore, each equality holds if and only if $G$ is a regular graph.
\end{theorem}

\begin{proof}
Lemma \ref{l:h} gives
$$
\frac{\d}{\sqrt{d_u d_v}}\leq \frac{2\sqrt{d_u d_v}}{d_u + d_v}\leq  \frac{\D}{\sqrt{d_u d_v}}\,.
$$
If $\a \ge -1/2$, then
$$
\d(d_ud_v)^{\a}\leq \D^{2\a+1} \frac{2\sqrt{d_u d_v}}{d_u + d_v}\,,
\qquad \d^{2\a+1} \frac{2\sqrt{d_u d_v}}{d_u + d_v}\leq  \D(d_ud_v)^{\a}.
$$
If $\a \le -1/2$, then
$$
\d(d_ud_v)^{\a}\leq \d^{2\a+1} \frac{2\sqrt{d_u d_v}}{d_u + d_v}\,,
\qquad \D^{2\a+1} \frac{2\sqrt{d_u d_v}}{d_u + d_v}\leq  \D(d_ud_v)^{\a}.
$$
We obtain the results by summing up these inequalities for $uv\in E(G)$.

If the graph is regular, then the lower and upper bounds are the same, and they are equal to $GA_1(G)$.
If the equality holds in the lower bound, then Lemma \ref{l:h} gives
$d_u = d_v = \d$ for every $uv\in E(G)$; hence, $d_u = \d$ for every $u\in V(G)$ and the graph is regular.
If the equality is attained in the upper bound, then Lemma \ref{l:h} gives
$d_u = d_v = \D$ for every $uv\in E(G)$ and we conclude $d_u = \D$ for every $u\in V(G)$.
\end{proof}

We would emphasize the following direct consequence of Theorem \ref{t:z1}. The upper bound was previously stated in \cite{RS2} and the lower bound in \cite{S}. 

\begin{corollary}
For any graph $G$, 
$$\d R(G) \leq GA_1(G)\leq \D R(G),$$
and each equality holds if and only if $G$ is regular.
\end{corollary}

\begin{theorem} \label{t:lb55}
Let $G$ be a graph  and $\a \in \mathbb{R} \setminus\{0\}$. Then the following statements hold. 
\begin{enumerate}[{\rm (a)}]
\item If $\a \le 1/2$, then $\d^{-2\a+1}\Delta^{-1} R_{\a}(G) \leq GA_1(G)\leq \Delta^{-2\a+1}\delta^{-1} R_{\a}(G)$.
\item $\a \ge 1/2$, then $\D^{-2\a} R_{\a}(G) \leq GA_1(G)\leq \d^{-2\a} R_{\a}(G) $.
\end{enumerate}
Furthermore, each equality holds if and only if $G$ is a regular graph.
\end{theorem}
\begin{proof}
Notice that
$$
GA_1(G)
= \sum_{uv\in E(G)}\frac{2\sqrt{d_u d_v}}{d_u + d_v}
=2  \sum_{uv\in E(G)} \frac{(d_u d_v)^{\a}(d_u d_v)^{-\a+1/2}}{d_u+d_v}
\,.
$$
Now, if $\a\le 1/2$, then $\d^{-2\a+1} \le (d_u d_v)^{-\a+1/2} \le \D^{-2\a+1}$, which implies  that
$$\d^{-2\a+1}\Delta^{-1} R_{\a}(G) \leq GA_1(G)\leq \Delta^{-2\a+1}\delta^{-1} R_{\a}(G).$$
Analogously, if $\a\ge 1/2$, then $\D^{-2\a+1} \le (d_u d_v)^{-\a+1/2} \le \d^{-2\a+1}$,
which implies that
$$\D^{-2\a} R_{\a}(G) \leq GA_1(G)\leq \d^{-2\a} R_{\a}(G) .$$
If the graph is regular, then the lower and upper bounds are the same, and they are equal to $GA_1(G)$.
If a bound is attained, then we have either
$d_ud_v = \d$ or $d_ud_v = \D$ for every $uv\in E(G)$, so that $G$ is a regular graph.
\end{proof}
It is readily seen that if $\alpha <0$, then Theorem \ref{t:z1} gives better results than Theorem \ref{t:lb55} and, if $\alpha >0$, then Theorem \ref{t:lb55} gives better results than Theorem \ref{t:z1}.

The well-known P\'{o}lya-Szeg\"{o} inequality can be stated as follows. 

\begin{lemma}{\rm \cite[p.62]{HLP}} \label{l:PS}
If $0<n_1 \le a_j \le N_1$ and $0<n_2 \le b_j \le N_2$ for $1\le j \le k$, then
$$
\left(\sum_{j=1}^k a_j^2 \right)^{1/2} \left(\sum_{j=1}^k b_j^2 \right)^{1/2}
\le \frac12 \left(\sqrt{\frac{N_1N_2}{n_1n_2}}+ \sqrt{\frac{n_1n_2}{N_1N_2}} \,\right)\left(\sum_{j=1}^k a_jb_j \right) .
$$
\end{lemma}

Theorems \ref{t:mz}, \ref{t:mz2} and \ref{t:mzz} will show the usefulness of P\'{o}lya-Szeg\"{o} inequality to deduce lower bounds on $GA_1$, as well as the usefulness of Cauchy-Schwarz inequality to deduce upper  bounds.

\begin{theorem} \label{t:mz}
For any graph $G$,
$$
\frac{2\D\d^2}{\D^2+\d^2} \, \sqrt{m R_{-1}(G)}\,\le GA_1(G) \le \D \sqrt{m R_{-1}(G)}\, ,
$$
and each equality holds if and only if $G$ is a regular graph.
\end{theorem}

\begin{proof}
First of all, Lemma \ref{l:h} gives
\begin{equation}\label{Eq-Lemma2.6}
\d \le \frac{2d_u d_v}{d_u + d_v} \le \D  .
\end{equation}
We also have
\begin{equation}
\frac{1}{\D} \le \frac{1}{\sqrt{d_u d_v}} \le \frac{1}{\d}\,.
\end{equation}

These inequalities and  P\'{o}lya-Szeg\"{o} inequality  give
$$
\begin{aligned}
GA_1(G) & = \sum_{uv\in E(G)}\frac{2\sqrt{d_u d_v}}{d_u + d_v}
\ge
\frac{\left(\sum_{uv\in E(G)} \frac{1}{d_ud_v} \right)^{1/2} \left(\sum_{uv\in E(G)} \frac{(2d_u d_v)^2}{(d_u + d_v)^2}  \right)^{1/2}}
{\frac12 \left(\frac{\D}{\d} + \frac{\d}{\D} \right)}
\\
& \ge \frac{2\D\d}{\D^2+\d^2} \, \sqrt{R_{-1}(G)}\,
\left(\sum_{uv\in E(G)} \d^2 \right)^{1/2}\\
&= \frac{2\D\d^2}{\D^2+\d^2} \, \sqrt{m R_{-1}(G)}
\,.
\end{aligned}
$$

Cauchy-Schwarz inequality gives
$$
\begin{aligned}
GA_1(G) & = \sum_{uv\in E(G)}\frac{2\sqrt{d_u d_v}}{d_u + d_v}
\le
\left(\sum_{uv\in E(G)} \frac{1}{d_ud_v} \right)^{1/2} \left(\sum_{uv\in E(G)} \frac{(2d_u d_v)^2}{(d_u + d_v)^2}  \right)^{1/2}
\\
& \le \sqrt{R_{-1}(G)}\,
\left(\sum_{uv\in E(G)} \D^2 \right)^{1/2}\\
&= \D \sqrt{m R_{-1}(G)}
\,.
\end{aligned}
$$

If the graph is regular, then the lower and upper bounds are the same, and they are equal to $GA_1(G)$.
If the equality holds in the lower bound, then the left hand side equality holds in \eqref{Eq-Lemma2.6}, so that Lemma \ref{l:h} gives
$d_u = d_v = \d$ for every $uv\in E(G)$; hence, $d_u = \d$ for every $u\in V(G)$ and the graph is regular.
Analogously, if the equality holds in the upper bound, then the right hand side equality holds in \eqref{Eq-Lemma2.6}, so that Lemma \ref{l:h} gives
$d_u = d_v = \D$ for every $uv\in E(G)$ and we can conclude that $d_u = \d$ for every $u\in V(G)$.
\end{proof}

\begin{theorem} \label{t:mz2}
For any graph $G$,
$$
 \frac{4\D^2\d^2\sqrt{2\d M_1(G)R_{-1}(G)}}{(\D^2 + \d^2)(\d + \D)^2}\le GA_1(G) \le \frac{ \sqrt{2\D M_1(G)R_{-1}(G)} }{2}\,,
$$
and each equality holds if and only if $G$ is a regular graph.
\end{theorem}

\begin{proof}
Notice that
$$
GA_1(G)= \sum_{uv\in E(G)}\frac{2\sqrt{d_u d_v}}{d_u + d_v}
\leq \sum_{uv\in E(G)}\frac{d_u + d_v}{2\sqrt{d_u d_v}}\,.
$$
Using the Cauchy-Schwarz inequality, we obtain
$$
\begin{aligned}
GA_1(G)
& \le \sum_{uv\in E(G)}\frac{d_u + d_v}{2\sqrt{d_u d_v}}\\
&\le \left(\sum_{uv\in E(G)} \frac{(d_u +d_v)^2}{4} \right)^{1/2} \left(\sum_{uv\in E(G)} \frac{1}{d_ud_v}\right)^{1/2}
\\
& \le \left(\frac{ \,\D}{2}\!\!\!\sum_{uv\in E(G)} (d_u +d_v) \right)^{1/2} \left(R_{-1}(G)\right)^{1/2}\\
&= \sqrt{\frac{\D M_1(G)R_{-1}(G)}{2}}\,.
\end{aligned}
$$

Let us prove the lower bound. 
Notice that
$$\frac{2\sqrt{d_u d_v}}{d_u + d_v}= \frac{4d_ud_v}{(d_u + d_v)^2}\frac{{d_u + d_v}}{2\sqrt{d_u d_v}}
\,.
$$
By Lemma \ref{c:t}, we have
$$
\frac{4\D\d}{(\D + \d)^2}\leq\frac{4d_ud_v}{(d_u + d_v)^2}\leq1.
$$
Since
$$
%\begin{aligned}
\frac{2\sqrt{d_u d_v}}{d_u + d_v}\geq \frac{4\d\D}{(\d + \D)^2} \frac{{d_u + d_v}}{2\sqrt{d_u d_v}}
\,,
%\end{aligned}
$$
we have
$$
GA_1(G)\geq \frac{4\d\D}{(\d + \D)^2} \sum_{uv\in E(G)}\frac{{d_u + d_v}}{2\sqrt{d_u d_v}}
\,.
$$
Since
$\d \le \frac{d_u + d_v}{2} \le \D,$
$\frac{1}{\D} \le \frac{1}{\sqrt{d_ud_v}} \le \frac{1}{\d}$,
the P\'{o}lya-Szeg\"{o} inequality gives

\begin{align*}
\sum_{uv\in E(G)}\frac{d_u + d_v}{2\sqrt{d_u d_v}}
&\ge
\frac{\left(\displaystyle\sum_{uv\in E(G)}\frac{(d_u+d_v)^{2}}{4} \right)^{1/2} \left(\displaystyle\sum_{uv\in E(G)}\frac{1}{d_ud_v} \right)^{1/2}}
{\frac12 \left(\frac{\D}{\d} + \frac{\d}{\D} \right)} \\
&\ge \frac{\D\d\sqrt{2\d M_1(G)R_{-1}(G)}}{\D^2 + \d^2}.
\end{align*}
Therefore,
$$
GA_1(G)\geq \frac{4\d\D}{(\d + \D)^2} \, \frac{\D\d\sqrt{2\d M_1(G)R_{-1}(G)}}{\D^2 + \d^2}
\,.
$$

If $G$ is a regular graph, then the lower and upper bounds are the same, and they are equal to $GA_1(G)$.
If we have the equality in the upper  bound, then $d_u + d_v = 2\D$  for every $uv\in E(G)$; hence, $d_u = \D$ for every $u\in V(G)$ and so the graph is regular. By analogy we can see that  the equality in the upper leads to  regularity of $G$.
%If we have the equality in the lower bound, then $d_u + d_v = 2\d$ for every $uv\in E(G)$, and $d_u = \d$ for every $u\in V(G)$.
\end{proof}

%Theorem \ref{t:lb5}, with $\a=1/2$, has the following consequence.
%
%
%\begin{corollary} \label{c:lb50}
%We have for any graph $G$
%$$
%\frac{2m^2}{(m+1)R(G)}
%\le GA_1(G)
%\le \frac{(\D+\d)^2 m^2}{4\D\d^2 R(G)} \,,
%$$
%and the upper bound is attained if and only if $G$ is regular.
%\end{corollary}
%
%
%Since $R(G) \le n/2$ (see \cite{RS}), Corollary \ref{c:lb50} has the following consequence.

\begin{theorem} \label{t:mzz}
For any graph $G$ and $\a> 0$,
$$
k_{\a} \sqrt{R_{\a}(G) R_{-\a}(G)} \le GA_1(G) \le \sqrt{R_{\a}(G) R_{-\a}(G)}\, ,
$$
with
$$
k_{\a}:=
\left\{
\begin{array}{ll}
\frac{2\D^{1/2}\d^{3/2}}{\D^{2}+\d^{2}} ,\quad &\mbox{if }\ 0 < \a\le 1,\\
\\
\frac{2\D^{\a-1/2}\d^{\a+1/2}}{\D^{2\a}+\d^{2\a}} ,\quad &\mbox{if }\ \a\ge 1,
\end{array}
\right.
$$
and  each inequality holds only if $G$ is a regular graph.
\end{theorem}

\begin{proof}
Cauchy-Schwarz inequality and Lemma \ref{c:t} give
$$
\begin{aligned}
GA_1(G) & = \sum_{uv\in E(G)}\frac{2\sqrt{d_u d_v}}{d_u + d_v}\\
&\le
\left(\sum_{uv\in E(G)} (d_ud_v)^{-\a} \right)^{1/2} \left(\sum_{uv\in E(G)} \frac{4d_u d_v(d_ud_v)^{\a}}{(d_u + d_v)^2}  \right)^{1/2}
\\
& \le \left(\sum_{uv\in E(G)} (d_ud_v)^{-\a} \right)^{1/2} \left(\sum_{uv\in E(G)} (d_ud_v)^{\a} \right)^{1/2}\\
&= \sqrt{R_{\a}(G) R_{-\a}(G)}
\,.
\end{aligned}
$$

Lemma \ref{l:h} gives
$$
\begin{aligned}
\d^{\a} \le \frac{2d_u d_v}{d_u + d_v}\,(d_u d_v)^{(\a-1)/2} \le \D^{\a}, \qquad & \text{ if }\,\a\ge 1,
\\
\d\D^{\a-1} \le \frac{2d_u d_v}{d_u + d_v}\,(d_u d_v)^{(\a-1)/2} \le \D\d^{\a-1}, \qquad & \text{ if }\,0< \a\le 1,
\end{aligned}
$$

If $\a\ge 1$, then these inequalities, $\D^{-\a} \le (d_u d_v)^{-\a/2} \le \d^{-\a}$
and Lemmas \ref{l:PS} and \ref{l:h} give

\begin{align*}
GA_1(G) & = \sum_{uv\in E(G)}\frac{\sqrt{d_u d_v}}{\frac12 (d_u + d_v)}\\
&\ge
\frac{\left(\sum_{uv\in E(G)} (d_ud_v)^{-\a} \right)^{1/2} \left(\sum_{uv\in E(G)} \frac{4(d_u d_v)^2}{(d_u + d_v)^2}\,(d_u d_v)^{\a-1}  \right)^{1/2}}
{\frac12 \left(\frac{\D^{\a}}{\d^{\a}} + \frac{\d^{\a}}{\D^{\a}} \right)}
\\
& = \frac{2\D^{\a}\d^{\a}}{\D^{2\a}+\d^{2\a}} \, \sqrt{R_{-\a}(G)}\, \left(\sum_{uv\in E(G)} \frac{2}{d_u + d_v} \,\frac{2d_u d_v}{d_u + d_v} \, (d_u d_v)^{\a} \right)^{1/2}
\\
& \ge \frac{2\D^{\a}\d^{\a}}{\D^{2\a}+\d^{2\a}} \, \sqrt{R_{-\a}(G)}\,
\left(\,\frac{\d}{\D}\sum_{uv\in E(G)} (d_u d_v)^{\a}\right)^{1/2}\\
&=  \frac{2\D^{\a-1/2}\d^{\a+1/2}}{\D^{2\a}+\d^{2\a}} \, \sqrt{ R_{\a}(G) R_{-\a}(G)}
\,.
\end{align*}

If $0 < \a\le 1$, then similar computations (using the bounds for $0 < \a\le 1$) give the lower bound.

\smallskip

If the graph is regular, then the two bounds are the same, and they are equal to $GA_1(G)$.
If the lower bound is attained, then Lemma \ref{l:h} gives
$d_u = d_v = \d$ for every $uv\in E(G)$ and we conclude $d_u = \d$ for every $u\in V(G)$.
If the lower bound is attained, then Lemma \ref{l:h} gives
$d_u = d_v = \D$ for every $uv\in E(G)$ and we conclude $d_u = \D$ for every $u\in V(G)$.
\end{proof}

In \cite[Theorem 4]{S} appear the inequalities
$$
\frac{2\d^2}{\D^2+\d^2} \sqrt{M_2(G) R_{-1}(G)}
\leq GA_1(G)
\le \sqrt{M_2(G) R_{-1}(G)} \, .
$$
Theorem \ref{t:mzz} generalizes these bounds.
Furthermore, the following consequence of Theorem \ref{t:mzz} (with $\a=1$) improves the lower bound above.

\begin{corollary} \label{c:mis29}
We have for any graph $G$
$$
\frac{2\d}{\D^2+\d^2} \, \sqrt{\d \D M_2(G) R_{-1}(G)}
\leq GA_1(G)
\le \sqrt{M_2(G) R_{-1}(G)} \, ,
$$
and the equality is attained if and only if $G$ is a regular graph.
\end{corollary}

The \emph{modified Narumi-Katayama index}
$$
NK^*= NK^*(G) = \prod_{u\in V (G)} d_u^{d_u} = \prod_{uv\in E (G)} d_u d_v
$$
was introduced in \cite{GSG}, inspired in the Narumi-Katayama index defined in \cite{NK} (see also \cite{G}, \cite{N}).
Next, we prove some inequalities relating the modified Narumi-Katayama index with others topological indices.

\begin{theorem} \label{t:nk3}
We have for any graph $G$ and $\a \in \mathbb{R} \setminus\{0\}$
$$
R_\a(G) \ge m \, NK^*(G)^{\a/m},
$$
and the equality holds if and only if   $(d_u d_v)$ has the same value for every $uv\in E(G)$.
\end{theorem}

\begin{proof}
Using the fact that the geometric mean is at most the arithmetic mean, we obtain
$$
\frac{1}{m}\,R_\a(G)
= \frac{1}{m}\,\sum_{uv\in E(G)} (d_u d_v)^\a
\ge \left(\, \prod_{uv\in E(G)} (d_u d_v)^\a \,\right)^{1/m}
= NK^*(G)^{\a/m}.
$$
The equality holds if and only if $(d_u d_v)$ has the same value for every $uv\in E(G)$.
\end{proof}

Theorems \ref{t:z1} and \ref{t:nk3} have the following consequence.

\begin{corollary} \label{c:nk3}
We have for any graph $G$ and $\a \in \mathbb{R} \setminus\{0\}$
$$
\begin{aligned}
GA_1(G)\geq \d^{-2\a} m \, NK^*(G)^{\a/m} ,
& \qquad \text{ if } \a \le -1/2,
\\
GA_1(G)\geq \d \D^{-2\a-1} m \, NK^*(G)^{\a/m} ,
& \qquad \text{ if } \a \ge -1/2,
\end{aligned}
$$
and the equality holds if and only if $G$ is regular.
\end{corollary}

 %\section*{References} 
\bibliographystyle{elsarticle-num}

\end{document}